\documentclass[11pt]{amsart}
\usepackage{graphicx}
\usepackage{amsmath,amsthm}
\usepackage{amssymb}
\usepackage{pinlabel}
\usepackage[left = 3cm, right = 3 cm , top = 3 cm , bottom = 3cm ]{geometry}
\usepackage[utf8]{inputenc}
\usepackage{todonotes}

\newtheorem{thm}{Theorem}[section]
\newtheorem{prop}[thm]{Proposition}
\newtheorem{lem}[thm]{Lemma}
\newtheorem{cor}[thm]{Corollary}

\newtheorem{obs}[thm]{Observation}
\newtheorem{question}[thm]{Question}

\theoremstyle{remark}
\newtheorem{remark}[thm]{Remark}

\theoremstyle{definition}
\newtheorem{definition}[thm]{Definition}

\begin{document}

\title{Lower bound for the Perron-Frobenius degrees of Perron numbers}
\author{Mehdi Yazdi}
\date{}
\thanks{Partially supported by NSF Grants DMS-1006553 and DMS-1607374.}

\maketitle

\begin{abstract}
Using an idea of Doug Lind, we give a lower bound for the Perron-Frobenius degree of a Perron number that is not totally-real, in terms of the layout of its Galois conjugates in the complex plane. As an application, we prove that there are cubic Perron numbers whose Perron-Frobenius degrees are arbitrary large; a result known to Lind, McMullen and Thurston. A similar result is proved for biPerron numbers.
\end{abstract}

\section{Introduction}

Let $A$ be a non-negative, integral, \emph{aperiodic} matrix, meaning that some power of $A$ has strictly positive entries. One can associate to $A$ a subshift of finite type with topological entropy equal to $\log(\lambda)$, where $\lambda$ is the spectral radius of $A$. By Perron-Frobenius theorem, $\lambda$ is a \emph{Perron} number \cite{gantmacher2005applications}; a real algebraic integer $ p \geq 1$ is called Perron if it is strictly greater than the absolute value of its other Galois conjugates. Lind proved a converse, namely any Perron number is the spectral radius of a non-negative, integral, aperiodic matrix \cite{lind1984entropies}. As a result, Perron numbers naturally appear in the study of entropies of different classes of maps such as: post-critically finite self-maps of the interval \cite{thurston2014entropy}, pseudo-Anosov surface homeomorphisms \cite{fried1985growth}, geodesic flows, and Anosov and Axiom A diffeomorphisms \cite{lind1984entropies}. 

Given a Perron number $p$, its \emph{Perron-Frobenius degree}, $d_{PF}(p)$, is defined as the smallest size of a non-negative, integral, aperiodic matrix with spectral radius equal to $p$. In other words, the logarithms of Perron numbers are exactly the topological entropies of mixing subshifts of finite type, and the Perron-Frobenius degree of a Perron number is the smallest `size' of a mixing subshift of finite type realising that number. Our main result gives a lower bound for the Perron-Frobenius degree of a Perron number, which is not totally-real. See the related work of Boyle-Lind, which gives an upper bound in the context of non-negative polynomial matrices \cite{boyle2002small}.

\begin{thm}
Let $p>0$ be a Perron number. Assume that some Galois conjugate $p'$ of $p$ is not real, and $ \hspace{3mm} \eta := \tan^{-1} \bigg( \frac{p - \textup{Re}(p')}{|\textup{Im}(p')|} \bigg) \leq 1 $. Then 
\[ d_{PF}(p) \geq \frac{2 \pi}{3 \eta}. \]
\label{lower-perron}
\end{thm}

To visualise the angle $\eta$ geometrically, see the left hand side of Figure \ref{Step1} for $t = \frac{p'}{p}$. It was known to Lind, McMullen \cite{McMullen} and Thurston (\cite[Note in Page 6]{thurston2014entropy}) that there are examples of Perron numbers of constant algebraic degree (in fact cubics), whose Perron-Frobenius degrees are arbitrary large. Their proofs are not published to the best of the author's knowledge. As the first application, we give a proof of their result.

\begin{cor}[Lind, McMullen, Thurston]
For any $N>0$, there are cubic Perron numbers whose Perron-Frobenius degrees are larger than $N$. 
\label{cubic}
\end{cor}

The second application is a similar result for a class of algebraic integers called \emph{biPerron} numbers. A unit algebraic integer $\alpha>1$ is called biPerron, if all other Galois conjugates of $\alpha$ lie in the annulus  $\{ z \in \mathbb{C} \hspace{2mm}| \hspace{2mm} \frac{1}{\alpha}< |z| < \alpha \}$, except possibly for $\alpha^{-1}$.

BiPerron numbers appear in the study of stretch factors of pseudo-Anosov homeomorphisms, in particular \emph{the surface entropy conjecture} (also known as Fried's conjecture). Fried proved that the stretch factor of any pseudo-Anosov homeomorphism on a closed, orientable surface $S$ is a biPerron number \cite{fried1985growth}, and Penner showed the Perron-Frobenius degree of the stretch factor is at most $6 |\chi(S)|$ (see \cite[Page 5]{penner1991bounds}). The strong form of the surface entropy conjecture states that the set of stretch factors of pseudo-Anosov homeomorphisms over all closed, orientable surfaces is exactly the set of biPerron numbers (see \cite[Problem 2]{fried1985growth} or \cite{McMullen}). In \S 3, we prove the following corollary.

\begin{cor}
For any $N>0$, there are biPerron numbers of algebraic degree $\leq 6$, whose Perron-Frobenius degrees are larger than $N$.
\label{biperron}
\end{cor}

We don't know if the examples in the above corollary arise as stretch factors of pseudo-Anosov maps (see Question \ref{PF-degree-question}).

\subsection{Outline} In Section \S 2, we recall the proof of Lind's theorem, and prove Theorem \ref{lower-perron}. In Section \S 3, two applications of the main theorem are proved, namely Corollaries \ref{cubic} and \ref{biperron}. In Section \S 4 further questions are suggested regarding Perron numbers arising as stretch factors of pseudo-Anosov homeomorphisms.

\subsection{Acknowledgement} I would like to thank Doug Lind for suggesting the idea of the proof of Theorem \ref{lower-perron} through a Mathoverflow post, and to thank Doug Lind and Curt McMullen for giving comments on an earlier version of this paper. The author acknowledges the support by a Glasstone Research Fellowship.

\section{Perron-Frobenius degree}

%
%


Given an algebraic integer $\lambda$ of degree $d$ over $\mathbb{Q}$ and minimal polynomial $f(x) = x^d - c_1 x^{d-1}- \dots - c_{d}$, define its companion matrix as 
\[ 
B = \left[\begin{array}{cccc}
0 & 0 & \dots & c_d  \\ 
1 & 0 & \dots &c_{d-1} \\
0 & 1 & \dots & c_{d-2} \\
\vdots & \vdots && \vdots \\
0 & 0 & \dots & c_1 \\
\end{array}\right] .
\]
Note that the characteristic polynomial of $B$ is equal to $f(x)$ up to sign. The Jordan form of $B$ shows that $\mathbb{R}^d$ splits into a direct sum of $1$-dimensional and $2$-dimensional $B$-invariant subspaces corresponding to real roots and pairs of conjugate complex roots of $f(x)$. If $\lambda'$ is a root of $f(x)$, denote the $B$-invariant subspace corresponding to $\lambda'$ by $E_{\lambda'}$, and let $\pi_{\lambda'} \colon \mathbb{R}^d \rightarrow E_{\lambda'}$ be the projection to $E_{\lambda'}$ along the complementary direct sum. As $\lambda$ is real, $E_{\lambda}$ is $1$-dimensional. Fixing a point $w \in E_{\lambda}$, we identify $rw $ with $r$ for $r \in \mathbb{R}$. Let $E$ be the \emph{positive half-space corresponding to $\lambda$}, that is the set of points such that their projection under $\pi_{\lambda}$ is a positive multiple of $w$. By an integral point in $E$ we mean an integral point with respect to the standard basis of $\mathbb{R}^d$.

\begin{thm}[Lind \cite{lind1984entropies}] 
Let $\lambda$ be a Perron number, with the companion matrix $B \colon\mathbb{R}^d \rightarrow \mathbb{R}^d$. Let $E$ be the positive half-space corresponding to $\lambda$. There are integral points $z_1 , \dots , z_n$ in $E$ such that for each $1 \leq i \leq n$, $Bz_i = \sum_{j=1}^{n} a_{ij}z_j$ with $a_{ij} \in \mathbb{N} \cup \{ 0 \}$, and any irreducible component of the matrix $A = [a_{ij}]$ is an aperiodic matrix whose spectral radius is equal to $\lambda$.
\label{Lind}
\end{thm}

The next theorem, also due to Lind, gives a converse to the previous theorem.

\begin{thm}[Lind \cite{lind1984entropies}] Let $\lambda$ be a Perron number, with the companion matrix $B \colon\mathbb{R}^d \rightarrow \mathbb{R}^d$. Let $E$ be the positive half-space corresponding to $\lambda$. If $A$ is an $n \times n$ aperiodic, non-negative, integral matrix with spectral radius equal to $\lambda$, then there are integral points $z_1 , \dots , z_n \in E$ such that for each $1 \leq i \leq n$ we have $Bz_i = \sum_{j=1}^{n} a_{ij}z_j$.
\label{Lindconverse}
\end{thm}

\noindent We recall Lind's proof of Theorem \ref{Lindconverse}. 

\begin{proof}
Consider $A: \mathbb{R}^n \longrightarrow \mathbb{R}^n$. By Perron-Frobenius theory, there is a \underline{positive} eigenvector $v \in \mathbb{R}^n$ corresponding to $\lambda$. By working over the field $\mathbb{Q}(\lambda)$, we can assume that $v \in \mathbb{Q}(\lambda)^n$. Let 
\[ v=(v_1, v_2, \dots , v_n)^T \in \mathbb{R}^n, \]
and assume that for $1 \leq i \leq n$:
\[ v_i = z_{i1}+z_{i2}\lambda+ \dots + z_{id}\lambda^{d-1} >0, \]
where the numbers $z_{ij}$ are integers. Define for $1 \leq i \leq n$:
\[ z_i = (z_{i1}, z_{i2}, \dots , z_{id})^T \in \mathbb{Z} ^d. \]
Since $v$ is an eigenvector for $A$
\[ \lambda v_i = (Av)_i = \sum_{j} a_{ij} v_j  \hspace{4mm}(*). \]
Let $\Psi : \mathbb{Q}(\lambda) \longrightarrow \mathbb{Q}^d$ be the map:
\[ \Psi(a_0 + a_1 \lambda + \dots + a_{d-1}\lambda^{d-1}) = (a_0 , \dots , a_{d-1})^T.  \]
In particular, $z_i = \Psi(v_i)$ for $1 \leq i \leq n$. Taking $\Psi$ from both side of the equation $(*)$ gives us: 
\[ \Psi(\lambda v_i) = \sum_{j} a_{ij} \Psi(v_j). \]
Note that multiplication by $\lambda$ on $\mathbb{Q}(\lambda)$ has matrix $B$ with respect to the basis $\{1,  \lambda, \dots , \lambda^{d-1} \}$. Hence, we obtain:
\[ Bz_i = \sum_{j} a_{ij}z_j .\]
Finally, we need to verify that the points $z_i$ belong to the positive half-space $E$. Note that 
\[ w^* = (1, \lambda, \dots, \lambda^{d-1}) \in \mathbb{R}^d, \]
is a \underline{left} eigenvector for the linear map $B$ corresponding to the eigenvalue $\lambda$. Let $E_\lambda$ be the one-dimensional invariant subspace of $\mathbb{R}^d$ corresponding to $\lambda$ and $C$ be its invariant complement. Therefore,  
\[ C = \{ x \in \mathbb{R}^d \hspace{2mm}| \hspace{2mm} w^*x  =0 \} . \]
Let $\pi_\lambda$ be the projection map from $\mathbb{R}^d$ onto $E_\lambda$ along the complementary direct sum. Define a map $m_{w^*}:  \mathbb{R}^d \longrightarrow \mathbb{R}$ that is multiplication by $w^*$ from the left. Then $m_{w^*}$ should be a multiple of the map $\pi_\lambda$. On the other hand 
\[  m_{w^*}(z_i) = w^*z_i  = z_{i1}+ z_{i2} \lambda+ \dots + z_{id} \lambda^{d-1}=v_i >0. \]
Hence replacing each $z_i$ by $-z_i$ if necessary (in case $m_{w^*}$ is a negative multiple of $\pi_\lambda$), we have $z_i \in E$ for each $i$ and the proof is complete.
\end{proof}

\begin{remark}
$\mathbb{R}^d$ can be identified with $\mathbb{Q}(\lambda)\otimes_{\mathbb{Q}} \mathbb{R}$. Multiplication by $\lambda$ is a linear map on $\mathbb{Q}(\lambda)\otimes_{\mathbb{Q}} \mathbb{R}$ which has the matrix $B$ with respect to the basis $\{ 1, \lambda, \cdots, \lambda^{d-1}\}$. Therefore an integral point in the standard basis of $\mathbb{R}^d$ can be considered as a point in $\mathbb{Z}[\lambda]$.
\end{remark}

The following lemma and propositions will be used in the proof of Theorem \ref{lower-perron}.

\begin{lem}
Let $\lambda$ be a Perron number, with the companion matrix $B \colon\mathbb{R}^d \rightarrow \mathbb{R}^d$. Let $\delta$ be a Galois conjugate of $\lambda$, and $\pi_{\delta}$ be the projection onto $E_{\delta}$ along the complementary direct sum. Then for any non-zero integral point $z \in \mathbb{R}^d$, $\pi_{\delta}(z) \neq 0$. 
\label{projection}
\end{lem}

\begin{proof}
 

Assume the contrary that $\pi_{\delta}(z) =0$. Therefore $z$ lies in the invariant complementary direct sum of $E_{\delta}$ in $\mathbb{R}^d$. Set $z = (x_1 , \cdots, x_d)^T \in \mathbb{Z}^d$. Working in the complexification $\mathbb{R}^d \otimes_{\mathbb{R}} \mathbb{C}$ of $\mathbb{R}^d$, we obtain that $w^* z = 0$ where $ w^* = [1, \delta, \cdots, \delta^{d-1}]$
is a \underline{left} eigenvector for $B$ corresponding to the eigenvalue $\delta$. Therefore
\[ x_1+ x_2 \delta + \dots + x_d \delta^{d-1} =0. \]
However, this means that $\delta$ satisfies an integral polynomial equation with degree less than $d$. Therefore all $x_j$ should be zero. This contradicts the fact that $z \in E$. 
\end{proof}

The idea of using the next proposition has been suggested generously by Douglas Lind in the Mathoverflow post https://mathoverflow.net/questions/228826/lower-bound-for-perron-frobenius-degree-of-a-perron-number. In this post, the author had asked for a way of finding a lower bound for the Perron-Frobenius degree of a Perron number. This was the answer that Lind gave:

``If a Perron number $\lambda$ has negative trace, then any Perron-Frobenius matrix must have size strictly greater than the algebraic degree of $\lambda$, for example the largest root of $x^3+3x^2-15x-46$. If $B$ denotes the $d \times d$ companion matrix of the minimal polynomial of $\lambda$ (which of course can have negative entries), then $\mathbb{R}^d$ splits into the dominant 1-dimensional eigenspace $D$ and the direct sum $E$ of all the other generalized eigenspace. 

Although I've not worked this out in detail, roughly speaking the smallest size of a Perron-Frobenius matrix for $\lambda$ should be at least as large as the smallest number of sides of a polyhedral cone lying on one side of $E$ (positive $D$-coordinate) and invariant (mapped into itself) under $B$. This is purely a geometrical condition, and there are likely further arithmetic constraints as well. For example, if $\lambda$ has all its other algebraic conjugates of roughly the same absolute value, then $B$ acts projectively as nearly a rotation, and this forces any invariant polyhedral cone to have many sides, so the geometric lower bound will be quite large.'' 

The following proposition is only one way of using the above idea and it would be nice to weaken the geometric assumptions about the roots or to explore the arithmetic constraints that Lind mentions.

\begin{prop}
Let $\hat{B} : \mathbb{R}^3 \longrightarrow \mathbb{R}^3$ be a linear map. Assume that the eigenvalues of $\hat{B}$ are $\lambda$, $\delta$ and $\theta$ such that
\begin{enumerate}
\item $\lambda >1$ is a positive real number, and $\delta, \theta$ is a pair of conjugate complex numbers with non-zero imaginary parts and positive real parts, and $|\delta|<\lambda$.
\item If we set $t = \frac{\delta}{\lambda}$, then
$ \hspace{2mm} \eta:=|\frac{1 - \textup{Re}(t)}{\textup{Im}(t)} | \leq 1$. 
\end{enumerate}

Define $E$ as the positive half-space corresponding to the eigenvalue $\lambda$. Let $M$ be the minimum number of sides for an arbitrary non-degenerate polygonal cone $\hat{\mathcal{C}}  \subset E$ that is invariant under the map $\hat{B}$, i.e., $ \hat{B}(\hat{\mathcal{C}}) \subset \hat{\mathcal{C}}$.
Then $M  \geq \frac{2\pi}{3\eta}$. 
\label{invariantcone}
\end{prop}

\begin{proof}
Let $\hat{\mathcal{C}}$ be an invariant polygonal cone for the map $\hat{B}$ with $M$ sides. Let $E_\lambda$ be the 1-dimensional invariant subspace in $\mathbb{R}^3$ corresponding to $\lambda$. Pick an eigenvector $w \in E$ corresponding to the eigenvalue $\lambda$, and let $H$ be the set of points whose projection under $\pi_{\lambda}$ is the constant vector $w$. Define 
\[ \mathcal{P}:= \hat{\mathcal{C}} \cap H. \]
Hence, $\mathcal{P}$ is a polygon with $M$ sides and $\hat{\mathcal{C}}$ is the cone over $\mathcal{P}$ (Figure \ref{cone}). 

\begin{figure}
\labellist
\pinlabel $\mathcal{P}$ at 130 150
\pinlabel $\hat{\mathcal{C}}$ at 190 80
\endlabellist
\centering 
\includegraphics[width= 1.5 in]{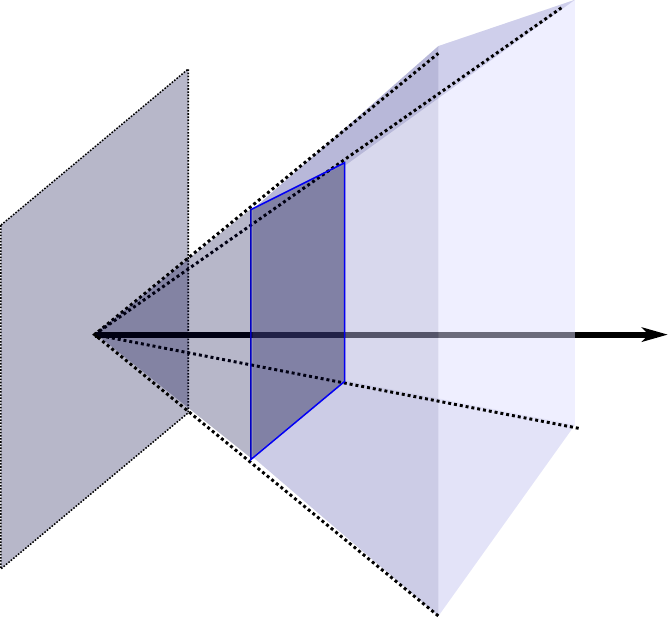}
\caption{The cone $\hat{\mathcal{C}}$ over the polygon $\mathcal{P}$.}
\label{cone}
\end{figure}

Let $G$ be the 2-dimensional invariant subspace corresponding to $\delta$. One can think about $G$ as the complex plane with the action of $\hat{B}$ on $G$ being the multiplication by the complex number $\delta$. 


Now if $w + w'$ is a vector in $H$ where $w' \in G$, then 
\[ \hat{B}(w+w') = \lambda w + \delta w' = \lambda(w+ \frac{\delta}{\lambda}w'). \]
Here by $\delta w'$ we mean multiplication by $\delta$ inside the complex plane $G$. Note that $w+ \frac{\delta}{\lambda}w' \in H$, hence the action of $\hat{B}$ on $H$ is the multiplication by the complex number $t = \frac{\delta}{\lambda}$. As a corollary, the polygon $\mathcal{P}$ is invariant under multiplication by $t = \frac{\delta}{\lambda}$. Note that $0 \in \mathcal{P}$ since $|t|= \frac{|\delta|}{\lambda}<1$ and successive multiplication by $t$ converges to the origin in $H$ (that is the intersection point $E_\lambda \cap H$). Now if we set $\eta =|\tan^{-1}(\frac{1 - \textup{Re}(t)}{\textup{Im}(t)}) |$, by Proposition \ref{convex} we have $M  \geq \frac{2\pi}{3\eta}$. 

Note this proposition is purely geometric and $\lambda$ does not need to be an algebraic integer.
\end{proof}

\begin{prop}
Let $\mathcal{P}$ be a convex non-degenerate polygon in the complex plane, having $M$ sides and containing the origin. Let $t$ be a complex number with non-zero imaginary part and positive real part. Assume that $\mathcal{P}$ is invariant under multiplication by $t$, i.e., $ t \mathcal{P} \subset \mathcal{P}$.
If $ \eta := |\tan^{-1}\bigg(\frac{1 -\textup{Re}(t)}{\textup{Im}(t)}\bigg)| \leq 1 $, then $ M \geq \frac{2 \pi}{3\eta}$.
\label{convex}
\end{prop}

\begin{proof}
Without loss of generality assume that $|t| \leq 1$ and $\textup{Im}(t)>0$. See Figure \ref{Step1} to visualize the angle $\eta$ geometrically. Let $P_1 , \dots , P_M$ be the vertices of $\mathcal{P}$ in counter clockwise order and let $O$ denote the origin. Define $\beta_j := \angle(OP_jP_{j+1})$ and $\phi_j := \angle(P_jOP_{j+1})$ (see Figure \ref{triangle}). The proof is divided into a few Steps.\\

\begin{figure}
\labellist
\pinlabel $P_j$ at 95 187
\pinlabel $P_{j+1}$ at -5 93
\pinlabel $O$ at 150 20
\pinlabel $\beta_j$ at 100 150
\pinlabel $\phi_j$ at 125 47
\endlabellist
\centering 
\includegraphics[width= 1.3 in]{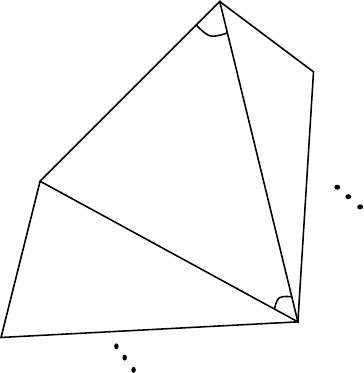}
\caption{The triangle $OP_jP_{j+1}$.}
\label{triangle}
\end{figure}

\textbf{Step 1}: \vspace{-5mm} \[ \beta_j \geq \frac{\pi}{2}-\eta. \]
This is because if the above condition is not satisfied, then $t P_j$ lies outside of the polygon $\mathcal{P}$ (see Figure \ref{Step1}, right hand side). Contradicting the assumption that the polygon $\mathcal{P}$ is invariant under multiplication by $t$.\\ 

\begin{figure}
\labellist
\pinlabel $t$ at 115 65
\pinlabel $\eta$ at 140 35
\pinlabel $1$ at 140 0
\pinlabel $0$ at -5 0
\pinlabel $O$ at 210 0
\pinlabel $P_j$ at 358 -5 
\pinlabel $tP_j$ at 345 65
\pinlabel $P_{j+1}$ at 285 80
\pinlabel {$\frac{\pi}{2}- \eta$} at 375 30
\pinlabel $\beta_j$ at 320 18

\endlabellist
\centering 
\includegraphics[width= 3 in]{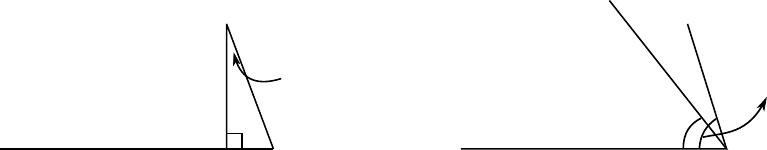}
\caption{Left: the angle $\eta$, Right: Step 1}
\label{Step1}
\end{figure}
Define $l_j = |OP_j|$. We will work with the values $\frac{l_{j+1}}{l_j}$. Note $P := \prod_{j=1}^M \frac{l_{j+1}}{l_j} =1$.  
The index set $\mathcal{M}: = \{ 1 , \dots , M \}$ can be partitioned into two sets according to whether $\phi_j < \eta$ or not. 
\[ \mathcal{A} = \{ 1 \leq j \leq M \hspace{2mm}| \hspace{2mm} \phi_j \geq \eta \}  \text{   and} \hspace{3mm} \mathcal{B}= \{ 1 \leq j \leq M \hspace{2mm}| \hspace{2mm} \phi_j < \eta \}.   \]
Define $ P_{\mathcal{A}}$ and $P_{\mathcal{B}}$ as follows
\[ P_\mathcal{A} := \prod_{j \in \mathcal{A}} \hspace{1mm} \frac{l_{j+1}}{l_j} \hspace{3mm}, \hspace{3mm}P_\mathcal{B} := \prod_{j \in \mathcal{B}} \hspace{1mm} \frac{l_{j+1}}{l_j}.\]
We clearly have $P = P_{\mathcal{A}} \hspace{1mm} \cdot \hspace{1mm} P_{\mathcal{B}}=1$. We give lower bounds for the values of $ P_{\mathcal{A}}$ and $P_{\mathcal{B}}$.\\
\textbf{Step 2}: \vspace{-5mm }\[ \phi_j -\eta < \frac{\pi}{2}. \]
To see this, consider the triangle $OP_jP_{j+1}$ and note that sum of any two angles has to be less than $\pi$. 
\[ \beta_j + \phi_j < \pi \implies \frac{\pi}{2}-\eta + \phi_j < \pi \implies \phi_j -\eta < \frac{\pi}{2}.\]
Here we used Step 1 for the first implication.\\
\textbf{Step 3}: For any $j \in \mathcal{A}$
\[ \frac{l_{j+1}}{l_j} \geq \frac{\cos(\eta)}{\cos(\phi_j -\eta)}.  \]
Consider the triangle $OP_jP_{j+1}$. Let $A$ be the point on the segment $OP_{j+1}$ such that $\angle(OP_jA) = \frac{\pi}{2}-\eta$. Such a point exists by Step 1, since 
\[ \angle{OP_jA = \frac{\pi}{2} - \eta \leq \beta_j = \angle{OP_j P_{j+1}}}.  \]
Let $H$ be the projection of $O$ onto $P_jA$ (see Figure \ref{convexlemma}). It follows from the assumption $j \in \mathcal{A}$ that the point $H$ lies inside the triangle $OP_jP_{j+1}$. This is because
\[ \angle{P_j OH} = \eta \leq \phi_j = \angle{P_j OA}. \]
Then 
\[ \frac{l_{j+1}}{l_j} = \frac{OP_{j+1}}{OP_j} \geq \frac{OA}{OP_j} = \frac{OA}{OH} \cdot \frac{OH}{OP_j} = \frac{1}{\cos(\phi_j -\eta)} \cdot \cos(\eta). \] 

\begin{figure}
\labellist
\pinlabel $H$ at 73 47
\pinlabel $P_j$ at 100 10
\pinlabel $P_{j+1}$ at 120 123
\pinlabel $A$ at 70 83
\pinlabel $O$ at -5 10
\pinlabel $\phi_j$ at 17 57

\pinlabel $O$ at 175 12
\pinlabel $P_j$ at 275 12
\pinlabel $P_{j+1}$ at 282 40
\pinlabel $H$ at 252 55
\pinlabel $D$ at 252 28
\pinlabel $\eta$ at 200 45
\pinlabel {$\frac{\pi}{2}-\eta$} at 250 -10
\pinlabel $\eta$ at 35 20
\endlabellist
\centering 
\includegraphics[width= 3.5 in]{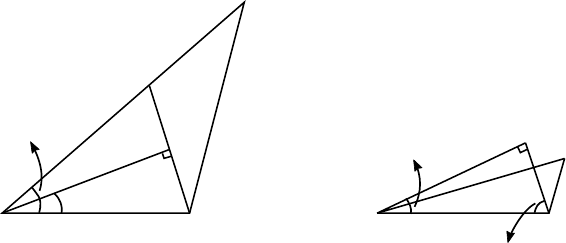}
\caption{Left: Step 3, Right: Step 4}
\label{convexlemma}
\end{figure}
\noindent \textbf{Step 4}: For any $j \in \mathcal{B}$, we have:
\[ \frac{l_{j+1}}{l_j} \geq \cos(\eta). \]
Choose the point $H$ such that $\angle(P_jOH) =\eta $ and $\angle(OP_jH) = \frac{\pi}{2}-\eta$. Therefore $\angle(OHP_j) = \frac{\pi}{2}$. Let $D$ be the intersection of the lines $OP_{j+1}$ with $P_jH$. Then $D$ lies on the segments $OP_{j+1}$ and $P_j H$ (see Figure \ref{convexlemma}). To see this note that 
\[ \angle{OP_jP_{j+1} }=\beta_j \geq \frac{\pi}{2} -\eta = \angle{OP_jH},  \]
\[ \angle{P_jOH} = \eta \geq \phi_j = \angle{ P_jOP_{j+1}} . \]
Here the first inequality is by Step 1, and the second inequality follows from the assumption $j \in \mathcal{\beta}$. Now 
\[ \frac{l_{j+1}}{l_j}= \frac{OP_{j+1}}{OP_j} \geq \frac{OD}{OP_j} \geq \frac{OH}{OP_j} = \cos(\eta).  \]
Now we are ready to give a lower bound for $M$. Note that Steps 3 and 4 imply that: 
\[ 1 = P = P_{\mathcal{A}} \hspace{1mm} \cdot \hspace{1mm} P_{\mathcal{B}} \geq \bigg( \prod_{j \in \mathcal{A}} \frac{\cos(\eta)}{\cos(\phi_j-\eta)}\bigg) \cdot \bigg( \prod_{j \in \mathcal{B} }\cos(\eta) \bigg) \]
\[ \implies \cos(\eta)^{\frac{M}{|\mathcal{A}|}} \leq \bigg( \prod_{j \in \mathcal{A}} \cos(\phi_j-\eta) \bigg) ^{\frac{1}{|\mathcal{A}|}}, \]
where $|\mathcal{A}|$ is the cardinality of $\mathcal{A}$. Now we observe that, keeping the sum of $\phi_j$ for $j \in \mathcal{A}$ fixed, the product of $\cos(\phi_j-\eta)$ is maximized when all $\phi_j$ are equal. This is simply a consequence of the following inequality, where we take $a$ and $b$ to be the quantities $\phi_j - \eta$: 
\[ \cos(a) \cdot \cos(b) \leq \bigg( \cos(\frac{a+b}{2}) \bigg)^2  \hspace{2mm}.   \]
Crucially $0 \leq \phi_j - \eta < \frac{\pi}{2}$, for each $j \in \mathcal{A}$ (by Step 2, and the definition of the set $\mathcal{A}$), which implies that all $\cos(\cdot)$ involved are non-negative. To see the inequality holds, note that:
\[ 2\cos(a) \cdot \cos(b) = \cos(a+b) + \cos(a-b) =   \]
\[=2 \big(\cos(\frac{a+b}{2})\big)^2 -1 + \cos(a-b) \leq  2 \big(\cos(\frac{a+b}{2})\big)^2. \]
Let $\overline{\phi}$ be the average of the angles $\phi_j -\eta $ for $j \in \mathcal{A}$; then we have $0 \leq \overline{\phi} \leq \frac{\pi}{2}$, since each of the angles $\phi_j - \eta$ satisfied the same bounds.  By definition of the set $\mathcal{B}$ 
\[ \forall j \in \mathcal{B} \hspace{3mm} \phi_j < \eta \implies   \sum_{j \in \mathcal{B}}\phi_j < \eta \cdot |\mathcal{B}| = \eta (M - |\mathcal{A}|). \]
Hence
\[ \bar{\phi} = \frac{\sum_{j \in \mathcal{A}} (\phi_j-\eta)}{|\mathcal{A}|}= \frac{\sum_{j \in \mathcal{M}}\phi_j-\sum_{j \in \mathcal{B}}\phi_j -|\mathcal{A}|\eta}{|\mathcal{A}|}\geq \]
\[ \geq \frac{2\pi-\eta(M-|\mathcal{A}|)-|\mathcal{A}|\eta}{|\mathcal{A}|}= \frac{2 \pi -M\eta}{|\mathcal{A}|}. \]
Now if $(2 \pi -M\eta)$ is negative, then there is nothing to prove. Otherwise $0 \leq \frac{2 \pi -M\eta}{|\mathcal{A}|} \leq \overline{\phi} \leq \frac{\pi}{2}$ and therefore: 
\[ \cos(\eta)^{\frac{M}{|\mathcal{A}|}} \leq \bigg( \prod_{j \in \mathcal{A}} \cos(\phi_j-\eta) \bigg) ^{\frac{1}{|\mathcal{A}|}} \leq \cos(\overline{\phi})  \leq \cos(\frac{2\pi-M\eta}{|\mathcal{A}|}). \]
The next step is to give a lower bound for $\cos(\eta)^{\frac{M}{|\mathcal{A}|}}$.\\
\textbf{Step 5}: For $\alpha \geq 1$ and $0 \leq x \leq \frac{1}{2}$ the following inequality holds: 
\[ (1-x)^\alpha \geq 1 -2 \alpha x. \]
This inequality can be proved by noting that the values of both sides agree at $x = 0$ and then checking the signs of derivatives for $0 \leq x \leq \frac{1}{2}$ and $1 \leq \alpha$ (for the variable $x$).\\ The assumption $\eta \leq 1$  implies that $0 \leq \frac{\eta ^2}{2} \leq \frac{1}{2}$ and hence $0 \leq 1 - \frac{\eta^2}{2}$.  The inequality $\cos(y) \geq 1-\frac{y^2}{2}$ holds for every real number $0 \leq y \leq 1$, and clearly $\frac{M}{|\mathcal{A}|} \geq 1$. Hence we have the following lower bound: 
\[ \cos(\eta)^{\frac{M}{|\mathcal{A}|}} \geq (1-\frac{\eta^2}{2})^\frac{M}{|\mathcal{A}|} \geq 1 - 2(\frac{M}{|\mathcal{A}|})\frac{\eta^2}{2} = 1 - \frac{M\eta^2}{|\mathcal{A}|}, \]
where in the last inequality we have used Step 5.
Combining with the previous bound, we obtain that: 
\[ 1 - \frac{M\eta^2}{|\mathcal{A}|} \leq \cos(\frac{2 \pi-M\eta}{|\mathcal{A}|}) \implies  1 - \cos(\frac{2\pi -M\eta}{|\mathcal{A}|}) \leq \frac{M\eta^2}{|\mathcal{A}|} \implies \]
\[ \implies 2\sin(\frac{2\pi-M\eta}{2|\mathcal{A}|})^2 \leq \frac{M\eta^2}{|\mathcal{A}|}  \implies \sin(\frac{2\pi-M\eta}{2|\mathcal{A}|}) \leq \sqrt{\frac{M}{2|\mathcal{A}|}} \hspace{1mm}\eta.  \]
Recall the assumption $0 \leq \frac{2\pi - M \eta}{|\mathcal{A}|} \leq \frac{\pi}{2}$. Therefore the quantity $\frac{2\pi - M \eta}{2|\mathcal{A}|}$ lies in the interval $[0, \frac{\pi}{4}]$. Now in the interval $[0,\frac{\pi}{4}]$, the inequality $\sin(x) \geq \frac{x}{\sqrt{2}}$ holds. Hence, 
\[ \frac{1}{\sqrt{2}} \frac{2\pi -M\eta}{2|\mathcal{A}|} \leq \sin(\frac{2\pi -M\eta}{2|\mathcal{A}|}) \leq \sqrt{\frac{M}{2|\mathcal{A}|}}\hspace{1mm} \eta. \]
\[ \implies 2\pi -M\eta \leq 2 \sqrt{M|\mathcal{A}|}\hspace{1mm} \eta  \implies 2\pi -M\eta \leq 2 M \eta \implies  \]
\[ \implies M\eta \geq \frac{2 \pi}{3}.  \]
\end{proof}

\begin{proof}[Proof of Theorem \ref{lower-perron}:]\ \\

Assume $d_{PF}(p) =n$. Therefore, there is an $n \times n$ non-negative, integral, aperiodic matrix $A=[a_{ij}]$ with spectral radius equal to $p$. Let $B \colon \mathbb{R}^d \rightarrow \mathbb{R}^d$ be the companion matrix corresponding to the minimal polynomial of $p$. Let $w$ be an eigenvector for the map $B$ corresponding to the eigenvalue $p$, and denote by $E \subset \mathbb{R}^d$ the positive half-space containing $w$.
By Theorem \ref{Lindconverse}, there are integral points $z_1, \dots, z_n \in E$ such that for each $1 \leq i \leq n$ 
\[ Bz_i = \sum a_{ij}z_j. \]
Let $\mathcal{C}$ be the cone over the points $z_1, \dots, z_n$, that is
\[ \mathcal{C} = \{ \epsilon_1 z_1 + \dots+\epsilon_n z_n \hspace{2mm}| \hspace{2mm} \forall i \hspace{2mm} \epsilon_i \geq 0  \} \subset E. \]
The cone $\mathcal{C}$ is invariant under the action of $B$, that is $ B(\mathcal{C}) \subset \mathcal{C}$.
Let $E_p$ and $E_{p'}$ be the 1-dimensional and 2-dimensional invariant subspaces of $\mathbb{R}^d$ corresponding to $p$ and $p'$ respectively. Set $W := E_p \oplus E_{p'}$, and let $\pi : \mathbb{R}^d \longrightarrow W $ be the projection onto the invariant subspace $W$ along the complementary direct sum. Since the maps $B$ and $\pi$ commute, we have:
\[ Bz_i = \sum_j a_{ij}z_j \implies B(\pi(z_i)) = \pi(B(z_i))=  \sum_j a_{ij}P(z_j).  \]
Therefore, if we set $\hat{z_i}:=\pi(z_i)$, then $\hat{z_i} \in W \cap E$ and they satisfy the same linear equations as $z_i$ did. Hence the cone $\hat{\mathcal{C}} \subset W \cap E$ over the points $\hat{z_i}$ is invariant under the linear action of $\hat{B}:=B_{|W}$.

By Lemma \ref{projection}, since $z_i$ are integral points, none of the points $\pi(z_i)$ can lie entirely inside the 1-dimensional subspace $E_p \subset W$; otherwise $\pi_{p'}(z_i)=0$. Therefore, the cone $\hat{\mathcal{C}}$ is non-degenerate. In summary, the cone over the points $\pi(z_i)$ is a non-degenerate polygonal cone $\hat{\mathcal{C}}$ in $W$, which is invariant under the action of $\hat{B}$. By assumption the map $\hat{B}$ satisfies the conditions of Proposition \ref{invariantcone}. Therefore, the desired bound holds. 

\end{proof}

\section{applications}

\newtheorem*{cubic}{Corollary \ref{cubic}}
\begin{cubic}
[Lind, McMullen, Thurston]
For any $N>0$, there are cubic Perron numbers whose Perron-Frobenius degrees are larger than $N$. 
\end{cubic}

\begin{proof}
The idea is to construct a cubic polynomial $f(x)$ with exactly one real root $w_1>0$, and such that for one of the other roots say $w_2$: 

\begin{enumerate}
\item The absolute value of $w_2$ is smaller but very close to $w_1$ and the argument of $w_2$ is very small. 
\item $f(x)$ is irreducible. 
\end{enumerate}
It is easy to construct a reducible polynomial of the form $g(x)=(c-x)[(a-x)^2+b^2]$ that satisfies $(1)$. Moreover by perturbing $g(x)$, one expects that $g(x)$ becomes irreducible and still satisfies (1).  The details are as follows. Let $\epsilon >0$. The proof breaks into a few parts. \\

\textbf{Step 1}: There are natural numbers $a,b,c \gg 0$ satisfying the inequalities
\begin{equation} 
\sqrt{a^2+b^2} < c \leq a+\epsilon \hspace{1mm} b, \hspace{6mm} \bigg(\frac{a}{b}\bigg)^2 \leq c. 
\label{step1}
\end{equation}

First by choosing $a_0$ much larger than $b_0$, we may arrange that $\sqrt{a_0^2+b_0^2}< a_0 + \epsilon \hspace{1mm}b_0$. 
Let $c_0$ be a positive integer satisfying $ \bigg( \frac{ a_0}{b_0} \bigg)^2 \leq c_0$. 
Pick $k \gg 0$ such that
\[ k(a_0+ \epsilon \hspace{1mm} b_0) - k \sqrt{a_0^2+b_0^2} \geq c_0+4, \]
and denote by $c$ the largest integer between $k \sqrt{a_0^2+b_0^2}$ and $k(a_0+ \epsilon \hspace{1mm} b_0)$. Therefore
\[ c \geq k (a_0 + \epsilon \hspace{1mm} b_0) -1 \geq (c_0+4) -1 = c_0 +3 > c_0.  \]
Set $a = k \hspace{1mm} a_0$ and $b= k \hspace{1mm} b_0$. Now we check that the desired inequalities hold for $a,b,c$. The first inequality is satisfied by the definition of $c$. Moreover
\[ c \geq c_0 \geq \bigg( \frac{a_0}{b_0} \bigg)^2 = \bigg( \frac{a}{b} \bigg)^2. \]
 
 Note that by choosing $k \gg 0$, we can assume that all of the numbers $a,b$ and $c$ are large. This proves Step 1. Define the cubic polynomial $f(x)$ as $ f(x) = (c-x)[(a-x)^2+b^2]+1$.\\

\textbf{Step 2}: The polynomial $f(x)$ has a real root $\omega_1$ satisfying
\[  c< \omega_1< \min \{ c+1, c+ \frac{c+1}{a^2+b^2-1} \}.  \]

By the Intermediate Value Theorem, it is enough to show that
\[ f(c)>0, \hspace{6mm} f(c+1)<0, \hspace{6mm} f(c+ \frac{c+1}{a^2+b^2-1} )<0. \]
We have $f(c)=1>0$, and 
\[ f(c+1) = -[(a-c-1)^2+b^2]+1 \leq -b^2+1 <0. \]
Here we have used the assumption $b>1$.

For the last part, set $p = c+ \frac{c+1}{a^2+b^2-1}$. Hence
\[ f(p)= - \frac{c+1}{a^2+b^2-1} [ (a-p)^2+b^2]+1 \leq - \bigg(\frac{c+1}{a^2+b^2-1}\bigg) \cdot b^2 +1 < 0 \iff \]
\[ \iff 1 < \bigg(\frac{c+1}{a^2+b^2-1}\bigg) \cdot b^2 \iff a^2+b^2-1 < c \hspace{1mm}b^2+b^2 \iff a^2-1 < c \hspace{1mm b^2}. \]
But the last inequality holds by the assumption $\big(\frac{a}{b}\big)^2 \leq c$. Therefore the Step follows.\\

\textbf{Step 3}: $f(x)$ has exactly one real root. \\

To see this, assume the contrary that all roots of $f(x)$ are real. Denote the other two roots by $\omega_2$ and $\omega_3$. We will prove that 
$ \bigg( \frac{\omega_2 + \omega_3}{2} \bigg)^2 < \omega_2 \omega_3$, 
which gives a contradiction. After expanding we deduce that 
\[ f(x) = -x^3 + (c+2a)x^2-(2ac+a^2+ b^2)x +c(a^2+b^2)+1 .\]
By the Vieta's formula
\[ \omega_1 + \omega_2 + \omega_3 = c+2a, \]
\[ \omega_1 \omega_2 \omega_3 = c(a^2+b^2)+1. \]
Therefore
\[ 0< 2a-1< \omega_2 + \omega_3 = c+2a- \omega_1 < 2a, \]
where we have used the inequality $c < \omega_1 < c+1$ from Step 2. As a result
\[ \omega_2 \omega_3 = \frac{c(a^2+b^2)+1}{\omega_1} > a^2 \iff \omega_1 < \frac{c(a^2+b^2)+1}{a^2}. \]
Here the first equality is the application of the Vieta's formula. Using Step 2, In order to verify the last inequality, it is enough to show that 
\[ c+ \frac{c+1}{a^2+b^2-1} < \frac{c(a^2+b^2)+1}{a^2}= c+ \frac{c \hspace{1mm } b^2+1}{a^2} \iff \]
\[ \iff \frac{c+1}{a^2+b^2-1} <  \frac{c \hspace{1mm } b^2+1}{a^2} \iff (c+1)a^2 < (c \hspace{1mm}b^2+1)(a^2+b^2-1). \]
But we have
\[ (c+1) < c\hspace{1mm} b^2+1, \hspace{6mm} a^2 < (a^2+b^2-1), \]
which imply the last part. Therefore, we established that $\omega_2 \omega_3 \geq a^2$. Putting them together, we obtain
\[ \bigg( \frac{\omega_2 + \omega_3}{2} \bigg)^2 < \bigg(\frac{2a}{2} \bigg)^2 = a^2 < \omega_2 \omega_3. \]
This completes the proof of the Step. Therefore, $\omega_2$ and $\omega_3$ are both non-real and $\omega_3 = \overline{\omega_2}$.\\

\textbf{Step 4}: $\omega_1$ is a Perron number.\\

Since $\omega_3 = \overline{\omega_2}$, we have $ |\omega_2|^2 = \omega_2 \omega_3$.
Therefore, by the Vieta's formula
\[ |\omega_2|^2 = \omega_2 \omega_3 = \frac{c(a^2+b^2)+1}{\omega_1} \leq \frac{c(a^2+b^2)+1}{c}= a^2+b^2 + \frac{1}{c} \leq a^2 +b^2+1 \leq c^2 < \omega_1^2. \]
Here we used $\omega_1 >c$ for the first and the last inequality. The relation $a^2+b^2+1 \leq c^2$ follows from $ a^2+b^2 < c^2$ (Step 1)
and the fact that both $a^2+b^2$ and $c^2$ are integers.\\

\textbf{Step 5}: 
\vspace{-4mm}
\[ |\omega_2|^2  \geq a^2+b^2-1.\]

Again using the Vieta's formula
\[ |\omega_2|^2 = \omega_2 \omega_3 = \frac{c(a^2+b^2)+1}{\omega_1} \geq a^2+b^2-1 \iff  \omega_1 \leq \frac{c(a^2+b^2)+1}{a^2+b^2-1} = c+ \frac{c+1}{a^2+b^2-1}. \]
But the last inequality holds by Step 2.\\

\textbf{Step 6}: Denote the real and imaginary part of $\omega_2$ by $\textup{Re}(\omega_2)$ and $\textup{Im}(\omega_2)$. Then
\[  |\textup{Re}(\omega_2) | \leq a, \hspace{4mm} 0 < \omega_1 - \textup{Re}(\omega_2) < c-a + 2, \hspace{4mm}  |\textup{Im}(\omega_2)|^2 \geq b^2-1. \]

Since $\omega_2$ and $\omega_3$ are complex conjugates, we have $\omega_2 + \omega_3 = 2 \textup{Re}(\omega_2)$.  
By the Vieta's formula
\[ \textup{Re}(\omega_2) = \frac{\omega_2+ \omega_3}{2} = \frac{\omega_1 + \omega_2+ \omega_3 - \omega_1}{2} = \frac{c+2a-\omega_1}{2}. \]
Therefore 
\[ |\textup{Re}(\omega_2)| \leq \frac{c+2a -c}{2} = a.   \]
Here for the last inequality, we have used $c< \omega_1 <c+1$ from Step 2. This verifies the part one of the Step. For the second part, by Step 4 $|\omega_2| < \omega_1$, which implies that $0< \omega_1 - \textup{Re}(\omega_2)$. Moreover 
\[ \omega_1 - \textup{Re}(\omega_2) = \omega_1 - \bigg(\frac{c+2a-\omega_1}{2} \bigg) = \frac{3 \omega_1 -(c+2a)}{2} \leq \frac{3(c+1)-(c+2a)}{2} < c-a+2.   \]
Here we used $\omega_1 < c+1$ from Step 2. This completes the second part. For the third part 
\[ |\textup{Im}(\omega_2)^2| = |\omega_2|^2 - |\textup{Re}(\omega_2)|^2 \geq  (a^2+b^2-1) - a^2 = b^2-1. \]
Here we have used Step 5, together with the part one of Step 6. \\

\textbf{Step 7}:  \vspace{-5mm}
\[ \bigg(\frac{\omega_1 - \textup{Re}(\omega_2)}{\textup{Im}(\omega_2)}\bigg)^2 \leq \frac{(\epsilon + 2 b^{-1})^2}{1 - b^{-2}}. \]

By parts two and three of Step 6
\[  \bigg(\frac{\omega_1 - \textup{Re}(\omega_2)}{\textup{Im}(\omega_2)}\bigg)^2 \leq \frac{(c-a+2)^2}{b^2-1} \leq \frac{(\epsilon \hspace{1mm}b +2)^2}{b^2-1}= \frac{(\epsilon + 2 b^{-1})^2}{1 - b^{-2}}.  \]
Here we have used the condition $0 < c-a \leq \epsilon \hspace{1mm}b $ from Step 1.\\

Now we can prove the Corollary. Consider the algebraic integer $\omega_1$ defined as above. Since by Step 2 $c < \omega_1 < c+1$, the number $\omega_1$ is not an integer. The other two roots of $f(x)$ are not real by Step 3. Hence $\omega_1$ is a cubic algebraic integer. By Step 4, $\omega_1$ is Perron. By Theorem \ref{lower-perron}
\[ d_{PF}(\omega_1) \geq \frac{2 \pi}{3 \eta}, \hspace{5mm} \eta:= \tan^{-1} \bigg( \frac{\omega_1 - \textup{Re}(\omega_2)}{|\textup{Im}(\omega_2)|} \bigg), \]
whenever $\eta \leq 1$. By Step 7 we have
\[ \tan(\eta) \leq \frac{\epsilon+ 2 b^{-1}}{\sqrt{1-b^{-2}}}.  \]
As mentioned in Step 1, given any $\epsilon>0$, we may find $a,b,c$ with the given properties such that they are arbitrary large. Therefore we may assume that $b> \epsilon^{-1}$, or equivalently $b^{-1}< \epsilon$. Hence
\[ \tan(\eta) \leq \frac{\epsilon+ 2 b^{-1}}{\sqrt{1-b^{-2}}} \leq \frac{3\epsilon}{\sqrt{1-b^{-2}}} < 6 \epsilon. \] 
Here the last inequality follows from $b \geq 2$.
To sum up we have $\tan(\eta) < 6 \epsilon$, which is equivalent to $\eta< \tan^{-1}(6\epsilon)$ since the tangent function is strictly increasing on the interval $[ 0 , \frac{\pi}{2}]$. As a result 
\[ d_{PF}(\omega_1) \geq \frac{2 \pi}{3 \eta} > \frac{2 \pi}{3 \tan^{-1}(6\epsilon)}. \]
By choosing $\epsilon>0$ arbitrary small, we find arbitrary large lower bounds for $d_{PF}(\omega_1)$. This completes the proof.
\end{proof}

\begin{remark}
If $\lambda$ is a quadratic Perron number, then $d_{PF}(\lambda)=2$. To see this, assume that the minimal polynomial of $\lambda$ is of the form $  f(x) = x^2 -u x+ v$, 
where $u,v \in \mathbb{Z}$ and $\Delta = u^2 - 4v >0$. If we denote the other root by $\lambda'$, then $ u = \lambda + \lambda' >0$, 
since $\lambda$ is Perron. Now if $u$ is even, then $4 | \Delta$ and we may take 
\[A =  \left[ \begin{array}{cc}
\frac{u}{2} & \frac{\Delta }{4}\\
1 & \frac{u}{2} \\
\end{array}
\right] . \]
Then all the entries of $A$ are positive integers, and its characteristic polynomial is equal to $f(x)$. If $u$ is odd, then $\Delta \equiv 1 \hspace{2mm}(\text{mod} \hspace{2mm} 4)$. Moreover $\Delta \neq 1$ since otherwise the polynomial $f(x)$ would not have been irreducible. Therefore, we may take
\[A =  \left[ \begin{array}{cc}
\frac{u+1}{2} & \frac{\Delta -1}{4}\\
1 & \frac{u-1}{2} \\
\end{array}
\right] . \]
The characteristic polynomial of $A$ is equal to $f(x)$. If $u >1$, then $A$ has positive entries. If $u =1$, then we should have $v <0$ since $\Delta>1$. In this case $A^2$ has positive entries.
$\square$
\end{remark}

An algebraic integer is called a \emph{unit}, if the product of all its Galois conjugates is equal to $1$ or $-1$. Equivalently, the constant term of its minimal polynomial should be equal to $1$ or $-1$.  

\begin{definition}
A unit algebraic integer $\alpha>1$ is called \textbf{biPerron}, if all other Galois conjugates of $\alpha$ lie in the annulus $\{ z \in \mathbb{C} \hspace{2mm}| \hspace{2mm} \frac{1}{\alpha}< |z| < \alpha \}$, except possibly for $\alpha^{-1}$.
\end{definition}

\begin{obs}
Let $\gamma>2$ be a Perron number, such that for every other Galois conjugate $\gamma'$ of $\gamma$ we have $ |\gamma'| \leq \gamma -2$.
Then the unique real solution $\alpha>1$ to  $\alpha+\frac{1}{\alpha}=\gamma$ is biPerron.
\label{perron-to-biperron}
\end{obs}

\begin{proof}
As the Galois conjugates of $\alpha$ come in reciprocal pairs, the product of all the Galois conjugates is equal to $1$. Therefore, $\alpha$ is a unit algebraic integer. Assume that $\alpha' \notin \{ \alpha , \alpha^{-1} \}$ is a Galois conjugate of $\alpha$. We need to prove that $\frac{1}{\alpha} < |\alpha'| < \alpha$. There is a Galois conjugate $\gamma' \neq \gamma$ of $\gamma$ such that $\alpha' + \frac{1}{\alpha'} = \gamma'$. There are three cases to consider: \begin{enumerate}
\item If $|\alpha'|>1$: 
By the triangle inequality 
\[ |\alpha'| \leq |\alpha'+ \frac{1}{\alpha'}| + |\frac{-1}{\alpha'}| = |\gamma'|+|\frac{1}{\alpha'}| \leq \gamma -2 + |\frac{1}{\alpha'}| = (\alpha + \frac{1}{\alpha})-2+ |\frac{1}{\alpha'}| < \alpha. \]
Here the last inequality follows from $|\alpha'|>1$ and $\alpha>1$. This proves the upper bound for $|\alpha'|$. The lower bound follows from $\frac{1}{\alpha}< 1 < |\alpha'|$.

\item If $|\alpha'|=1$: In this case the inequalities hold trivially as $\frac{1}{\alpha} < 1 = |\alpha'|=1 < \alpha$. 
\item If $|\alpha'|<1$: Then $(\alpha')^{-1}$ is also a Galois conjugate. The result follows from (1), since the inequalities are symmetric.
%
\end{enumerate}
\end{proof}

\begin{remark}
Without the condition on the absolute value of $\gamma'$, the conclusion is not true. As an example one can take $\gamma$ to be the Perron root of the polynomial $ (x-5)[(x-4)^2+3^2]-1$. 
\end{remark}

\newtheorem*{biperron}{Corollary \ref{biperron}}
\begin{biperron}
For any $N>0$, there are biPerron numbers of algebraic degree $\leq 6$, whose Perron-Frobenius degrees are larger than $N$.
\end{biperron}

\begin{proof}
Pick $\epsilon >0$. Let $\omega_1$ be the cubic Perron number constructed in Corollary \ref{cubic}, with Galois conjugates $w_2 = \overline{\omega_3}$. Recall that in the construction, one could take $a,b$ and $c$ to be arbitrary large. Therefore, we may assume that $b,c>2$ and $b > \epsilon^{-1}$. Throughout, when we refer to Step x, we mean Step x in the proof of Corollary \ref{cubic}. \\

By Step 2, $\omega_1 > c>2$, 
so the first condition of Observation \ref{perron-to-biperron} is satisfied. To prove the second condition, we want to show that it is possible to choose $a,b,$ and $c$ such that $ |\omega_2| \leq \omega_1 -2$. By the proof of Step 4, we have $ |\omega_2|^2 \leq a^2+b^2+ 1$. As we know that $c < \omega_1$, it is enough to prove that $a^2 + b^2 +1 \leq (c-2)^2$. In the proof of Step 1, we defined $c$ as the largest integer between $k \sqrt{a_0^2+b_0^2}$ and $k(a_0 + \epsilon \hspace{1mm} b_0)$. Moreover we had 
\[ k(a_0 + \epsilon \hspace{1mm} b_0) - k \sqrt{a_0^2+b_0^2} \geq c_0 + 4 \geq 4. \]
Therefore
\[ c \geq k \sqrt{a_0^2+b_0^2}+3 = \sqrt{a^2+b^2}+3.  \]
This implies 
\[ c-2 \geq \sqrt{a^2+b^2}+1 \implies (c-2)^2 \geq \bigg( \sqrt{a^2+b^2}+1 \bigg)^2 > a^2+b^2+1. \]
This shows that the hypotheses of Observation \ref{perron-to-biperron} are satisfied. Hence we may define the biPerron number $\alpha>1$ as the solution to  $\alpha+ \frac{1}{\alpha} = \omega_1$. The number $\alpha$ satisfies a monic integral polynomial equation of degree $6$, obtained by substituting $x = \alpha + \alpha^{-1}$ in the minimal polynomial of $\omega_1$, $f(x)$, and clearing the denominators. Hence the algebraic degree of $\alpha$ is at most $6$. In fact the degree can be taken to be equal to $6$ but we do not prove it. \\

The last step is to give a lower bound for the Perron-Frobenius degree of $\alpha$. Pick a Galois conjugate $\alpha' \neq \alpha$ of $\alpha$ such that $|\alpha'| \geq 1$. 
By Theorem \ref{lower-perron}, we have 
\[ d_{PF}(\alpha) \geq \frac{2 \pi}{3 \hat{\eta}}, \hspace{3mm} \hat{\eta} := \tan^{-1} \bigg( \frac{\alpha - \textup{Re}(\alpha')}{|\textup{Im}(\alpha')|} \bigg), \]
as long as $\hat{\eta} \leq 1$. We have 
\[ \alpha + \frac{1}{\alpha} = \omega_1 \implies \alpha = \omega_1 - \frac{1}{\alpha} <  \omega_1 , \]
since $\alpha>0$. Moreover 
\[\alpha' + \frac{1}{\alpha'} = \omega_2 \implies \textup{Re}(\alpha') = \textup{Re}(\omega_2 ) - \textup{Re}(\frac{1}{\alpha'}) \geq \textup{Re}(\omega_2) -1. \]
Here we have used the fact that $|\frac{1}{\alpha'}|\leq 1$, which implies that $|\textup{Re}(\frac{1}{\alpha'})| \leq 1$. Putting the two inequalities together, we obtain 
\[ \alpha - \textup{Re}(\alpha') \leq \omega_1 - \textup{Re}(\omega_2)+1.  \]
On the other hand, $\alpha$ is biPerron, so
\[ 0 < \alpha - |\alpha'| \leq \alpha - \textup{Re}(\alpha'). \]
Similarly 
\[ \alpha'+ \frac{1}{\alpha'} = \omega_2 \implies |\textup{Im}(\alpha')| = |\textup{Im}(\omega_2) - \textup{Im}(\frac{1}{\alpha'})| \geq |\textup{Im}(\omega_2)| - |\textup{Im}(\frac{1}{\alpha'})| \geq |\textup{Im}(\omega_2)|-1. \]
Now we can give an upper bound for $\tan(\hat{\eta})$:
\[ \tan(\hat{\eta}) = \bigg( \frac{\alpha - \textup{Re}(\alpha')}{|\textup{Im}(\alpha')|} \bigg) \leq \frac{\omega_1 - \textup{Re}(\omega_2)+1}{|\textup{Im}(\omega_2)|-1}.  \]
Using parts two and three of Step 6 together with Step 1, we obtain
\[ \tan(\hat{\eta}) \leq \frac{c-a+3}{\sqrt{b^2-1}-1} \leq \frac{\epsilon \hspace{1mm}b+3}{\sqrt{b^2-1}-1} = \frac{\epsilon+3 b^{-1} }{\sqrt{1 - b^{-2}}-b^{-1}}. \]
The condition $b > \epsilon^{-1}$ is equivalent to $b^{-1}<\epsilon$. We have
\[ \frac{\epsilon+3 b^{-1} }{\sqrt{1 - b^{-2}}-b^{-1}} \leq \frac{4 \epsilon}{\sqrt{1 - b^{-2}}-b^{-1}} \leq 16 \epsilon, \]
where the last inequality follows from $b \geq 2$. Therefore $\hat{\eta} \leq \tan^{-1}(16 \epsilon)$, which implies that 
\[ d_{PF}(\alpha) \geq \frac{2 \pi}{3 \hat{\eta}} \geq \frac{ 2 \pi}{3 \tan^{-1}(16 \epsilon)}. \]
By choosing $\epsilon >0$ to be arbitrary small, we obtain arbitrary large lower bounds for the Perron-Frobenius degree.
\end{proof}

\section{Questions}

Let $\lambda$ be the stretch factor of a pseudo-Anosov map $\phi$ on a closed orientable surface $S$. Then by Fried's theorem, $\lambda$ is biPerron. In fact one can construct a non-negative, integral, aperiodic matrix of size at most $6 |\chi(S)|$ and with spectral radius $\lambda$ using an invariant train track for the map $\phi$ \cite{penner1991bounds}. The author's initial motivation for studying the Perron-Frobenius degree of $\lambda$ was to control the genus of the underlying surface. Unfortunately our bound is not very effective for biPerron numbers coming from pseudo-Anosov maps, since `generic conjugacy class of pseudo-Anosov maps tend to have totally-real stretch factors' (see \cite{hamenstadt2018typical} or \cite[Appendix C5]{hubbard2016teichmuller} for a precise statement). This motivates the following question.

\begin{question}
Give an effective lower bound for the Perron-Frobenius degree of a, possibly totally-real, Perron number.
\end{question}

\begin{question}\

\begin{enumerate}
\item Are there pseudo-Anosov stretch factors with constant algebraic degree, and with arbitrary large Perron-Frobenius degree? In particular, can the biPerron numbers constructed in Corollary \ref{biperron} be realised as stretch factors? 
\item Fix a closed, orientable surface $S$. What is the set of possible Perron-Frobenius degrees of pseudo-Anosov maps on $S$? 
\end{enumerate}
\label{PF-degree-question}
\end{question}

Note a positive answer to Question (1) above, gives new counterexamples to a conjecture/question of Farb recently disproved by Leininger and Reid using methods from Teichm\"{u}ller theory \cite{leininger2017pseudo}.

\bibliographystyle{plain}
\bibliography{references}

\end{document}